\documentclass[letterpaper,11pt]{article}

\usepackage[utf8]{inputenc}
\usepackage[dvipsnames]{xcolor}
\usepackage[pdfencoding=auto, psdextra, colorlinks,citecolor=blue!90!black]{hyperref}       %
\usepackage{url}            %
\usepackage{booktabs}       %
\usepackage{amsfonts}       %
\usepackage{nicefrac}       %
\usepackage{mathrsfs}
\usepackage{graphicx}
\usepackage{amsmath,amssymb,amsthm}
\usepackage{bbold}
\usepackage[sans]{dsfont}
\usepackage{cleveref}
\usepackage{pifont}
\usepackage{algorithm,algorithmic}
\usepackage{thmtools}
\usepackage{thm-restate}
\def\O#1{\text{\ding{\the\numexpr#1+171}}}
\declaretheorem[name=Theorem,within=section]{theorem}
\declaretheorem[name=Lemma,sibling=theorem]{lemma}
\declaretheorem[name=Proposition,sibling=theorem]{proposition}
\declaretheorem[name=Fact,sibling=theorem]{fact}

\declaretheorem[name=Remark,sibling=theorem]{remark}
\declaretheorem[name=Definition,sibling=theorem]{definition}

\usepackage{bm}

\def\k{\textnormal{Ker}}

\def\dist{\textnormal{dist}}

\def\conv{\textnormal{Conv}}
\def\supp{\textnormal{supp}}
\def\spn{\textnormal{span}}

\usepackage{mathtools}
\usepackage{tcolorbox}
\usepackage{nicefrac}

\usepackage[numbers]{natbib}

\title{No Dimension-Free Deterministic Algorithm Computes \\ Approximate Stationarities of Lipschitzians}

\date{July 14, 2022}

\author{Lai~Tian\thanks{Department of Systems Engineering and Engineering Management, The Chinese University of Hong Kong, Sha Tin, N.T., Hong Kong SAR. E-mail: \href{mailto:tianlai@se.cuhk.edu.hk}{\tt tianlai@se.cuhk.edu.hk}.} \and
Anthony~Man-Cho~So\thanks{Department of Systems Engineering and Engineering Management, The Chinese University of Hong Kong, Sha Tin, N.T., Hong Kong SAR. E-mail: \href{mailto:manchoso@se.cuhk.edu.hk}{\tt manchoso@se.cuhk.edu.hk}.} 
}

\usepackage{geometry}
\def\intt{\mathop{\textnormal{int}}}
\def\bd{\mathop{\textnormal{bdry}}}

\usepackage{physics}
\usepackage{amsmath}
\usepackage{tikz}
\usepackage{mathdots}
\usepackage{yhmath}
\usepackage{cancel}
\usepackage{color}
\usepackage{siunitx}
\usepackage{array}
\usepackage{multirow}
\usepackage{amssymb}
\usepackage{gensymb}
\usepackage{tabularx}
\usepackage{extarrows}
\usepackage{booktabs}
\usetikzlibrary{fadings}
\usetikzlibrary{patterns}
\usetikzlibrary{shadows.blur}
\usetikzlibrary{shapes}

\begin{document}
\maketitle

\begin{abstract}
	We consider the computation of an approximately stationary point for a Lipschitz and semi-algebraic function $f$ with a local oracle. 
	If $f$ is smooth, simple deterministic methods have dimension-free finite oracle complexities.
	For the general Lipschitz setting, only recently, \citet{zhang2020complexity} introduced a randomized algorithm that computes Goldstein's approximate stationarity \citep{goldstein1977optimization} to arbitrary precision with a dimension-free polynomial oracle complexity. 
	
	In this paper, we show that no deterministic algorithm can do the same.
	Even without the dimension-free requirement, we show that any finite time guaranteed deterministic method cannot be general zero-respecting, which rules out most of the oracle-based methods in smooth optimization and any trivial derandomization of \citet{zhang2020complexity}. 
	Our results reveal a fundamental hurdle of nonconvex nonsmooth problems in the modern large-scale setting and their infinite-dimensional extension.
\end{abstract}

\section{Introduction}

Convexity and/or differentiability have been the safe haven for optimization and its applications for decades. Nowadays, nonconvex nonsmooth models (``non''-setting for short) are pervasive, e.g., ReLU neural networks, Generative Adversarial Network (GAN), Piecewise Affine Regression \cite{cui2021modern}, in modern machine learning, operations research, and statistics.
 In such a ``non''-era, automatic differentiation with PyTorch/TensorFlow may not be correct \citep{kakade2018provably}, subgradient flow is not necessarily convergent \cite{daniilidis2020pathological}, even the stationarity concepts are not trivial at all \cite{li2020understanding}.

In this paper, we consider the following  problem for an $L$-Lipschitz function $f:\mathbb{R}^d\rightarrow \mathbb{R}$:
\[\label{eq:p}
\min_{\bm{x}\in\mathbb{R}^d} f(\bm{x}), 
\]
where $f$ could be both nonsmooth and nonconvex. 
For such a general setting,
one of the arguably most fundamental questions to be asked is whether a stationary point of $f$ is computable and if so, how.
Note that when $f$ is smooth, it is folkloric that computing an $\epsilon$-stationary point (i.e., $\|\nabla f(\bm{x})\|\leq \epsilon$) only requires $O(\epsilon^{-2})$ calls to the gradient oracle with gradient descent \citep{nemirovskij1983problem}. 
Extensive efforts have been devoted to the fast computation of approximately stationary points for smooth $f$ in various settings \cite{ghadimi2013stochastic,ghadimi2016accelerated,ge2015escaping,jin2018accelerated,jin2021nonconvex,agarwal2017finding}. Lower bound of dimension-free complexity for smooth $f$ is also rather well-understood \cite{carmon2019lower,carmon2021lower}.
 However, as shown by \cite[Proposition 1]{kornowski2021oracle}, that computing  elements in $\{\bm{x}: \dist(0,\partial f(\bm{x})) \leq \epsilon\}$ in the ``non''-setting\footnote{Here $\partial f(\bm{x})$ is the Clarke subdifferential of $f$. See \Cref{def:subc} for details.} is impossible for any finite-time randomized/deterministic algorithm interacting with a local oracle. 
 
 For well-behaved problems with $\rho$-weakly convex\footnote{Recall $f(\bm{x})$ is $\rho$-weakly convex if $f(\bm{x}) + \frac{\rho}{2} \|\bm{x}\|^2$ is convex.} objective functions, \citet{davis2019stochastic,davis2019proximally} introduced a nice notion named near-approximate stationarity (NAS, see \Cref{def:nas}), which is closely related to the gradient norm of the Moreau envelope of $f$. They showed that a subgradient-type method computes an $(\epsilon,\delta)$-NAS point with dimension-free $O(\rho^4\delta^{-4}+\epsilon^{-4})$ calls to the subgradient oracle. However, \citet{kornowski2021oracle} proved that the oracle complexity of any randomized/deterministic algorithm for $(\epsilon,\delta)$-NAS cannot avoid an exponential dependence on the dimension if $f$ is only $L$-Lipschitz, which implies that the computation of NAS is in general intractable.

On the other front,
a notion that dates back to the seminal work of \citet{goldstein1977optimization}, termed Goldstein approximate stationarity (GAS, see \Cref{def:gas}), exhibits favorable algorithmic consequences. The conceptual scheme makes use of a $\delta$-approximation of the Clarke subdifferential $\partial_\delta f(\bm{x})$ (see \Cref{def:subg}). If we update iteratively with
\[
\bm{x}^{(t+1)} \leftarrow \bm{x}^{(t)} - \delta \cdot \bm{g}^{(t)} / \|\bm{g}^{(t)}\|,
\]
where $\bm{g}^{(t)} \coloneqq \arg\min_{\bm{g} \in \partial_\delta f(\bm{x}^{(t)})} \|\bm{g}\|$ is the minimal norm element in $\partial_\delta f(\bm{x}^{(t)})$, then an $(\epsilon,\delta)$-GAS point can be computed in $O(\epsilon^{-1}\delta^{-1})$ steps. 
However, obtaining $\bm{g}^{(t)}$ for a general Lipschitz function equipped with a local oracle can be intractable as there is no known approach to evaluate $\partial_\delta f(\bm{x})$. Therefore, a series of works, e.g., \citep{burke2020gradient,burke2005robust,kiwiel2007convergence} proposed to build a polyhedral approximation of $\partial_\delta f(\bm{x}^{(t)})$ via random sampling and compute an approximate $\bm{g}^{(t)}$ by solving a quadratic program in every iteration. However, the number of sampling points needed for meaningful approximation of $\partial_\delta f(\bm{x}^{(t)})\subseteq \mathbb{R}^d$ is lower bounded by the dimension $d$. Thus, a dimension-free finite-time complexity cannot be achieved within the existing gradient sampling scheme. 

Recently, \citet{zhang2020complexity} introduced a novel randomized algorithm that computes $(\epsilon,\delta)$-GAS points for $L$-Lipschitz functions with probability at least $1-\gamma$ and with a dimension-free oracle complexity 
\[
{O}\left(\frac{\Delta L^2}{\epsilon^3\delta} \log\left(\frac{\Delta}{\gamma\epsilon\delta}\right)\right),
\] 
where $f(\bm{0}) - \inf_{\bm{x}} f(\bm{x}) \leq \Delta$.
Specifically, the randomization in \citet{zhang2020complexity} appears when they sample $\bm{z}\in\mathbb{R}^d$ uniformly from a line segment $[\bm{x},\bm{y}]$ to exploit the fundamental theorem of calculus:
\[
\mathbb{E}_{\bm{z}}\Big[f'(\bm{z}; \bm{x}-\bm{y})\Big] = \int_0^1 f'\big( t\bm{x} + (1-t)\bm{y}; \bm{x}-\bm{y}\big) \mathrm{d} t = f(\bm{x}) - f(\bm{y}),
\]
which enables the computation of $\tilde{\bm{g}}^{(t)}: \langle \tilde{\bm{g}}^{(t)}, \bm{z}\rangle \leq \frac{1}{4} \| \tilde{\bm{g}}^{(t)} \|^2$ for $\bm{z} \in \partial_\delta f(\bm{x}^{(t)})$ with high probability. 
Therefore, GAS can be computed by a randomized algorithm to arbitrary precision with a dimension-free polynomial oracle complexity.

In sum, it is very curious to ask whether \citep{zhang2020complexity} can be derandomized or not. This could have both theoretical and practical impacts to the black-box optimization of Lipschitz functions, and potentially deepen our understanding of computability of various approximate stationaries.

\subsection{Our Results and Techniques.}

We show in \Cref{thm:find} that there exists an absolute constant $c$ such that for any $0\leq \epsilon,\delta < c$:
\begin{center}
\fbox{
No deterministic algorithm computes $(\epsilon,\delta)$-GAS with dimension-free complexity.}
\end{center}
This puts the dimension-free computation of GAS into a situation similar to that of computing the volume of a convex body \citep{dyer1991computing,barany1987computing,dyer1988complexity,dyer1991random}, for which randomization yields strict improvement. It also reveals a fundamental hurdle of nonconvex nonsmooth problems in the modern large-scale setting and their infinite-dimensional extension.

If we drop the dimension-free requirement and allow any deterministic algorithm with a finite oracle complexity (potentially dependent exponentially on dimension), we show in \Cref{thm:findzr} that there exists an absolute constant $c$ such that for any $0\leq \epsilon,\delta < c$:
\begin{center}
\fbox{
Any deterministic finite-time algorithm for $(\epsilon,\delta)$-GAS cannot be general zero-respecting.}
\end{center}
The notion \emph{general zero-respecting} (see \Cref{sec:setting})
generalized the \emph{zero-respecting} assumption from smooth optimization \cite[Section 2.2]{carmon2019lower} to the nonsmooth setting, which contains the classic notion of linear span \cite[Assumption 2.1.4]{nesterov2018lectures} as a special case.
This result rules out any trivial derandomization of \citet{zhang2020complexity} and most of the oracle-based methods in smooth optimization, e.g., gradient descent (with and without Nesterov acceleration), conjugate gradient \cite{hager2006survey}, BFGS and L-BFGS \cite{liu1989limited}, Newton’s method (with and without cubic regularization \cite{nesterov2006cubic}), and trust-region methods \cite{conn2000trust}.

The major obstacle in lower bounding the oracle complexity of GAS computation is the lack of hardness source.
In the smooth setting, almost all the hard constructions \cite{carmon2019lower,carmon2021lower,ghadimi2013stochastic} are built upon what Nesterov called ``the worst-function in the world'' \citep[\S 2.1.2]{nesterov2018lectures}. However, these constructions fail to rule out $(\epsilon,\delta)$-GAS if the iteration number is $ \Omega\big(\log (1/\delta)\big)$.
In the nonsmooth case, simple resisting oracle type constructions \cite{vavasis1993black,zhang2020complexity} would not work for $(\epsilon,\delta)$-GAS if there exist $i \in [T], j \in [T]$ such that  $0< \|\bm{x}_i - \bm{x}_j \| < \delta$.
Another source of hardness called ``string guessing'' \cite{bockenhauer2014string,braun2017lower} is popular in the online and nonsmooth convex settings. With a modified ``string guessing'' function, \citet{kornowski2021oracle} proved that the oracle complexity of any randomized/deterministic algorithm for $(\epsilon,\delta)$-NAS cannot avoid an exponential dependence on the dimension. However, these constructions would also be inapplicable to the computation of $(\epsilon,\delta)$-GAS if the iteration number is $ \Omega\big(\log (1/\delta)\big)$.

Our main technical contribution is a new resisting oracle type ``wedge''-shaped hard construction that is tailored for deterministic algorithms and GAS computation. The most interesting property of the construction is that the ``ambiguity region''  vanishes as the iterations get closer and closer.
On a high level, within the ``ambiguity region'', the algorithm cannot decide whether the underlying function is a single coordinate resisting function similar to \cite{vavasis1993black,zhang2020complexity} or our ``wedge'' construction. With very careful design and analysis, we eliminate all GAS points below certain precision near the  ``ambiguity region'', so that the algorithm cannot even identify a bounded set containing any GAS point. This summarizes the high-level idea of our new construction. We also remark that our hardness results hold even under a very strong local oracle assumption, which would return (generalized) derivatives of all orders (if exist).
\subsection{Related Work.}

\paragraph{Asymptotic Analysis.}
Clarke stationary points (i.e., $\{\bm{x}:0\in\partial f(\bm{x})\}$) are computable in the asymptotic regime for quite general functions. 
\citet{benaim2005stochastic,majewski2018analysis,davis2020stochastic} studied the asymptotic convergence of subgradient-type methods from a  differential inclusion perspective. Specifically, \citet{davis2020stochastic} showed the asymptotic convergence of subgradient method to Clarke stationary points for Whitney stratifiable objective functions. \citet{daniilidis2020pathological} introduced a pathological Lipschitz function for which the vanilla subgradient method may not converge even in continuous time. An interesting discussion of the relation between our impossibility results and these asymptotic analysis can be found in \Cref{rmk:asym}.
\paragraph{Nonasymptotic Analysis.}
The nonasymptotic analysis for the general ``non''-problem is still in its infancy stage. For the computation of NAS,
\citep{davis2019proximally,davis2019stochastic} showed that for $\rho$-weakly convex functions, $(\epsilon,\delta)$-NAS is computable with  $O(\rho^4\delta^{-4}+\epsilon^{-4})$ oracle calls. On the negative side,
\citet{kornowski2021oracle} showed that neither deterministic nor randomized algorithm can compute NAS for Lipschitz functions without an exponential dependence on dimension, which implies NAS is in general intractable. \citet{tian2021hardness} sharpened the hardness results for NAS to $\rho$-weakly convex with unbounded $\rho$, thus matching the positive results. For GAS, the gradient sampling scheme \cite{burke2005robust,kiwiel2007convergence,kiwiel2010nonderivative,burke2020gradient} promises finite but dimension-dependent complexity. Recently, \citet{zhang2020complexity} reported a randomized algorithm with a dimension-free oracle complexity to compute arbitrarily  precise GAS. However, \citep{zhang2020complexity} use an impractical  subgradient oracle, which is further  removed by extra randomized  procedures in \cite{tian2022complexity,davis2021gradient}.

\paragraph{Notation.} Throughout this paper, scalars, vectors and matrices are denoted by lowercase letters, boldface lower case letters and boldface uppercase letters, respectively.
The notation used in this paper is mostly standard in optimization and variational analysis.
 $\dist(\bm{x},S)\coloneqq\inf_{\bm{v}\in S} \|\bm{v}-\bm{x}\|$; $A \otimes B$ denotes the direct product of $A$ and $B$; $A^{\otimes 2}\coloneqq A\otimes A$; $\mathbb{B}_\epsilon(\bm{x})\coloneqq\{\bm{v}:\|\bm{v} - \bm{x}\|\leq \epsilon\}; \mathbb{B}\coloneqq \mathbb{B}_1(\bm{0})$; we may write $\mathbb{B}^d_\epsilon(\bm{x})$ to emphasize the dimension; $[\bm{x},\bm{y}]\coloneqq \{\gamma\bm{x}+(1-\gamma)\bm{y}:\gamma \in [0,1]\}$;  $\conv S$ is the convex hull of set $S$; $A^c$ is the complement of set $A$;  $\supp(\bm{x}) = \{i:x_i \neq 0\}$; for any optimization algorithm $A$ applied to a function $f$, we write the generated sequence as $\left\{\bm{x}^{A[f],(t)}\right\}_t$; we use $\bm{e}_i$ for the $i$-th column of identity matrix; $a\vee b\coloneqq	\max\{a,b\}$; $a\wedge b\coloneqq	\min\{a,b\}$; $\mathbb{N}^+\coloneqq \mathbb{N}\backslash\{0\}$.

\paragraph{Organization.} We introduce the necessary background on variational analysis and formal definitions of approximate stationarities in \Cref{sec:prel}. 
The main impossibility results and proofs are in \Cref{sec:main}. We conclude the paper in \Cref{sec:concl}.
\section{Preliminaries}\label{sec:prel}

\subsection{Generalized Differentiation Theory}

For a Lipschitz continuous function $f$ that could be both nonsmooth and nonconvex, a widely used generalized subdifferential is the  Clarke subdifferential \citep[Theorem 9.61]{rockafellar2009variational}:
\begin{definition}[Clarke subdifferential]\label{def:subc}
	\label{def:subd} Given a point $\bm{x}$, the Clarke subdifferential of Lipschitz $f$ at $\bm{x}$ is defined by
	\[
	\partial f(\bm{x}) \coloneqq \conv\big\{\bm{s}:\exists \bm{x}^\prime\rightarrow \bm{x}, \nabla f(\bm{x}^\prime) \textnormal{ exists}, \nabla f(\bm{x}^\prime)\rightarrow  \bm{s}\big\}.
	\]
\end{definition}
Perturbation and approximation are powerful principles underlying many optimization theory and algorithms.
The following $\delta$-approximation of $\partial f(\bm{x})$ introduced by \citet[Definition 2.2]{goldstein1977optimization} 
has a nice limiting behavior (see \Cref{fct:subd}) and is convenient for algorithmic developments.
\begin{definition}[Goldstein $\delta$-subdifferential]
	\label{def:subg} Given a point $\bm{x}$ and $\delta \geq 0$, the Goldstein $\delta$-subdifferential of Lipschitz $f$ at $\bm{x}$ is defined by
	\[
	\partial_\delta f(\bm{x}) \coloneqq \conv\big\{\textstyle{\bigcup_{\bm{y}\in\mathbb{B}_\delta (\bm{x})}} \partial f(\bm{y}) \big\}.
	\]
\end{definition}
Some useful properties of the Clarke subdifferential and its Goldstein approximation for Lipschitz continuous functions are collected below: 
\begin{fact}[cf. \citet{clarke1990optimization,goldstein1977optimization}]\label{fct:subd} For an $L$-Lipschitz continuous $f$ and $\delta > 0$,
\begin{itemize}
	\item $\partial f(\bm{x}), \partial_\delta f(\bm{x})$ are nonempty, convex, compact;
	\item $\partial f(\bm{x}) = \cap_{\delta > 0} \partial_\delta f (\bm{x})$;
	\item if $f$ is $C^1$ near $x$, then $\partial f(x) = \{\nabla f(\bm{x})\}$;
	\item if $f$ is convex, then $\partial f(x)$ is equal to the convex subdifferential \cite[\S D, Definition 1.2.1]{hiriart2004fundamentals}.
\end{itemize}
\end{fact}

\subsection{Approximate Stationarity Concepts}
We are now ready to provide the formal definitions of two important approximate stationarity notions, i.e., GAS \cite{goldstein1977optimization,zhang2020complexity,tian2022complexity} and NAS \cite{davis2019stochastic,davis2019proximally}. 
\begin{definition}[Goldstein approximate stationarity, GAS]\label{def:gas}
	Given a locally Lipschitz function $f:\mathbb{R}^d\rightarrow \mathbb{R}$, we say that $\bm{x}\in\mathbb{R}^d$ is an $(\epsilon,\delta)$-GAS point if
	\[
	\dist \Big(0, \partial_\delta f(\bm{x})\Big) \leq \epsilon.
	\]
\end{definition}

\begin{definition}[near-approximate stationarity, NAS]\label{def:nas}
	Given a locally Lipschitz function $f:\mathbb{R}^d\rightarrow \mathbb{R}$, we say that $\bm{x}\in\mathbb{R}^d$ is an $(\epsilon,\delta)$-NAS point if
	\[
	\dist \Big(0, \textstyle{\bigcup_{\bm{y}\in\mathbb{B}_\delta (\bm{x})}} \partial f(\bm{y})\Big) \leq \epsilon.
	\]
\end{definition}

It is easy to see that if $\bm{x}$ is NAS, then $\bm{x}$ is also GAS as $\partial_\delta f(\bm{x}) \supseteq\cup_{\bm{y}\in\mathbb{B}_\delta (\bm{x})} \partial f(\bm{y})$. However, the converse does not hold in general, even for convex \cite[Proposition 2.7]{tian2022complexity} and continuously differentiable functions \citep[Proposition 2]{kornowski2021oracle}. Besides, \citet{kornowski2021oracle} proved that the oracle complexity of any randomized/deterministic algorithm for $(\epsilon,\delta)$-NAS cannot avoid an exponential dependence on the dimension if $f$ is only $L$-Lipschitz, which implies that NAS is in general intractable.

\subsection{Existence and Impossibility of Testing}
In this subsection, we discuss the existence of and impossibility of testing GAS.

\begin{proposition}[existence]
Let $f:\mathbb{R}^d \rightarrow \mathbb{R}$ be local Lipschitz with $\inf f$ finite. Then for any $\epsilon > 0, \delta > 0$, there exists $\bm{x}$ such that 
\[
\dist\Big(0, {\textstyle \bigcup_{\bm{y} \in \mathbb{B}_\delta (\bm{x})} \partial f (\bm{y})} \Big) \leq \epsilon.
\] 
\end{proposition}
\begin{proof}
Let $\bm{x} \in \epsilon\delta$-$\arg\min f$, whose existence is guaranteed by finite $\inf f$.
	By Ekeland's variational principle \cite[Proposition 1.43]{rockafellar2009variational}, there exists a $\bm{y} \in \mathbb{B}_\delta(\bm{x})$ with $f(\bm{y}) \leq f(\bm{x})$ and $\arg\min_{\bm{z}} \{ f(\bm{z}) + \epsilon \|\bm{z} - \bm{y}\| \} = \{\bm{y}\}$. By Fermat's rule \cite[Theorem 10.1]{rockafellar2009variational} and sum rule \cite[Corollary 10.9]{rockafellar2009variational}, we have
	\[
	0\in \partial (f + \epsilon \|\cdot - \bm{y}\|) (\bm{y}) \subseteq \partial f(\bm{y}) + \epsilon \mathbb{B},
	\]
	which implies $\dist\big(0,\partial f(\bm{y})\big) \leq \epsilon$ as required.
\end{proof}

In the following theorem, we will use notions named \emph{deterministic algorithm} and \emph{local oracle}, whose formal definitions can be found in \Cref{sec:setting}.

\begin{theorem} Suppose that $0<\epsilon,\delta <1$.
For any deterministic algorithm $A$ using local oracle information, there exists a $2$-Lipschitz function $f:\mathbb{R} \rightarrow \mathbb{R}$ and $y\in\mathbb{R}$ such that $A$ cannot decide $\dist(0,\partial_\delta f(y)) \leq \epsilon$ or not with finite oracle complexity.
\end{theorem}
\begin{proof}
	Let $f_1(x) = x$ and $y=0$. Suppose that $A$ will return the correct answer, i.e., $\dist(0,\partial_\delta f_1(0)) > \epsilon$ and $A$ queries $\{x^{(t)}\}_{t=1}^T$ to the local oracle. Then, we only need to show there exists a function $f_2$ that is equal to $f_1$ in a neighborhood of $x^{(t)}$ for any $t \in [T]$ but $\dist(0,\partial_\delta f_2(0)) \leq \epsilon$. Indeed, such a construction is easy and similar to \citep[Section 5]{nesterov2012make} but not identical. 
	It is clear that there exists a line segment $[a,b] \subseteq \mathbb{B}_\delta (0)$ such that $x^{(t)} \notin [a,b],\forall t \in [T]$. Then, let 
	\[
	f_2(x) = \min\left\{x, a+ 4\left| x - \frac{b-a}{2} \right|  \right\}.
	\]
	Thus, $\|f_2\|_{\textnormal{Lip}} \leq 4$ and $0 \in \partial_\delta f(0)$ as required.
\end{proof}

\section{Deterministic Inapproximability of Stationarities}\label{sec:main}
In this section, we present the main results of this paper.
We discuss the formal definitions of oracle, algorithm, and functions in \Cref{sec:setting}. 
Then, the main theorems for general zero-respecting and general deterministic setting are reported in \Cref{sec:thms}. All proofs are collected in \Cref{sec:prfs}.
\subsection{Settings}\label{sec:setting}

\paragraph{Local Oracle.}
Given $F:\mathbb{R}^d\rightarrow \mathbb{R}$ and queried on $\bm{x}\in\mathbb{R}^d$, a local oracle $\mathcal{O}_F(\bm{x})$ returns a function $G:\mathbb{R}^d\rightarrow\mathbb{R}$ such that there exists $ \nu > 0:$
\[
F(\bm{y}) = G(\bm{y}), \qquad \forall \bm{y} \in \mathbb{B}^d_{\nu} (\bm{x}).
\]
\begin{remark}
A subtle but crucial point is that the local oracle $\mathcal{O}_F$ only returns the local copy function $G$ but not the radius $\nu$.
Otherwise, the resisting oracle argument in \Cref{sec:construction} would fail if the algorithm queries $\bm{x}^{(t+1)} \in \mathbb{B}_{\nu^{\bm{x}^{(t)}}} \left(\bm{x}^{(t)}\right)$. Nevertheless, $\mathcal{O}_F$ is still very powerful. If $F$ is smooth, then $O_F(\bm{x})$ is capable of providing (if exists)
$F(\bm{x})$, $\nabla F(\bm{x})$, $\nabla^2 F(\bm{x})$, and $\nabla^p F(\bm{x}), \forall p \in \mathbb{N}^+$.
For a nonsmooth $F$, $O_F(\bm{x})$ is capable of providing (if exists) the Clarke subdifferential \cite[Theorem 9.61]{rockafellar2009variational}, Fr\'{e}chet subdifferential \cite[Exercise 8.4]{rockafellar2009variational}, Mordukhovich limiting subdifferential \cite[Theorem 8.3(b)]{rockafellar2009variational}, and even the impractical subgradient selection oracle in \cite[Assumption 1(a)]{zhang2020complexity}. We note here that assuming a (unreasonably) strong oracle would only strength our impossibility results as the algorithms are allowed to use more information.
\end{remark}

We now turn to the formal definitions of the deterministic $\mathcal{A}_\textnormal{det}$ and deterministic general zero-respecting $\mathcal{A}_\textnormal{det-gzr}$ algorithm classes:
\paragraph{Algorithm Class.} We consider $\mathcal{A}_\textnormal{det}$ and  $\mathcal{A}_\textnormal{det-gzr}$, where
\begin{itemize}%
	\item $\mathcal{A}_\textnormal{det}$: all algorithms that use local information of current and past points from $\mathcal{O}_F(\bm{x})$ deterministically. 
Formally, for any $A\in\mathcal{A}_\textnormal{det}$, if there exists a $\nu > 0$ such that $\mathcal{O}_{F}\left(\bm{x}^{A[F],(i)}\right)(\bm{y}) = \mathcal{O}_{G}\left(\bm{x}^{A[F],(i)}\right)(\bm{y}), \forall i \in [t], \bm{y} \in \mathbb{B}_\nu\left( \bm{x}^{A[F],(i)}\right)$, then $\bm{x}^{A[F],(j)} = \bm{x}^{A[G],(j)},\forall j \in [t+1]$.
	\item  $\mathcal{A}_\textnormal{det-gzr}$: deterministic general zero-respecting algorithms satisfy $\mathcal{A}_\textnormal{det-gzr}\subseteq\mathcal{A}_\textnormal{det}$ and $\forall t \in \mathbb{N}^+$:
	\[
	\supp\Big(\bm{x}^{(t)}\Big) \subseteq \bigcup_{i<t} \left\{
	j \in [d]: \forall \nu > 0, \exists (\bm{y}, \bm{y}+\theta\bm{e}_j)  \in \mathbb{B}^d_{\nu} \left(\bm{x}^{(i)}\right)^{\otimes 2}, \theta \in \mathbb{R}:  F(\bm{y}) \neq F(\bm{y}+\theta\bm{e}_j )
	\right\}.
	\]
\end{itemize}

\begin{remark}
The oracle complexity of any deterministic algorithm interacting with a $p$th-order oracle $\big(f(\bm{x}),\nabla f(\bm{x}),\dots, \nabla^p f(\bm{x})\big)$ in the smooth setting \cite{carmon2019lower} is lower bounded by that of $\mathcal{A}_\textnormal{det}$ interacting with a local oracle.
General zero-respecting class is a nonsmooth generalization of the zero-respecting assumption \cite[Section 2.2]{carmon2019lower} from smooth optimization. Zero-respecting class contains most of the oracle-based methods in smooth optimization \cite{carmon2019lower}, e.g., gradient descent (with and without Nesterov acceleration), conjugate gradient \cite{hager2006survey}, BFGS and L-BFGS \cite{liu1989limited}, Newton’s method (with and without cubic regularization \cite{nesterov2006cubic}), and trust-region methods \cite{conn2000trust}. It also contains the widely used notion of linear span \cite[Assumption 2.1.4]{nesterov2018lectures} for developing lower bounds as a special case.

\end{remark}

\paragraph{Function class.} For a given $C>0$ and dimension $d \in \mathbb{N}^+$, we consider the following Lipschitz function class:
\[
\mathcal{F}^\textnormal{Lip}_{C,d}\coloneqq \left\{f:\mathbb{R}^{d'} \rightarrow \mathbb{R}: d'\in [d], \|f\|_\text{Lip} \leq C, f(\bm{0}) - \inf_{\bm{x}} f(\bm{x}) \leq C\right\}.
\]

\begin{remark}\label{rmk:fncs}
Our ``wedge'' hard construction is a Lipschitz continuous piecewise linear function. Therefore, the results in \Cref{sec:main} also hold for Lipschitz semi-algebraic functions \cite{davis2020stochastic}. 
\end{remark}

\subsection{Main Results}\label{sec:thms}
For the general deterministic setting, we have the following impossibility result:
\begin{theorem}[deterministic]\label{thm:find}
Suppose that $0<\epsilon, \delta < \frac{1}{\sqrt{17}}$ and $C\geq 6$. For any $T < +\infty$ and $d \geq T + 1$, we have
\[
\inf_{A \in \mathcal{A}_\textnormal{det}} \sup_{f \in \mathcal{F}_{C,d}^{\textnormal{Lip}}}\  \min_{t \in [T]} \ \dist\Big(0, \partial_\delta f\left(\bm{x}^{{A[f]},(t)} \right)\Big) > \epsilon.
\]	
\end{theorem}
\Cref{thm:find} shows that randomization is provably helpful in the dimension-free computation of GAS. It also reveals a fundamental hurdle of nonconvex nonsmooth problems in the modern large-scale setting and their infinite-dimensional extension.

Without the dimension-free requirement, we have the following impossibility result:
\begin{theorem}[deterministic general zero-respecting]\label{thm:findzr}
Suppose that $0<\epsilon, \delta < \frac{1}{\sqrt{17}}$ and $C\geq 6$.
For any $T < +\infty$ and $d \geq 2$, we have
\[
\inf_{A \in \mathcal{A}_\textnormal{det-gzr}} \sup_{f \in \mathcal{F}_{C,d}^{\textnormal{Lip}}}\  \min_{t \in [T]} \ \dist\Big(0, \partial_\delta f\left(\bm{x}^{{A[f]},(t)} \right)\Big) > \epsilon.
\]	
\end{theorem}
\Cref{thm:findzr} points out that any finite-time deterministic method for GAS must be significantly different from most of the commonly used algorithmic scheme in smooth optimization. Thus, even for finite-time computation of GAS, new algorithmic ideas are necessary.
\begin{remark}\label{rmk:asym}
It is notable that the hard construction in the proof of \Cref{thm:findzr} is semi-algebraic (see \Cref{rmk:fncs}) and \citet{davis2020stochastic} showed that every limiting point of the vanilla subgradient method is a Clarke stationary point for any semi-algebraic $f$, i.e., $\limsup_{t\rightarrow +\infty} \{\bm{x}^{(t)} \}\subseteq\{\bm{x}: 0 \in \partial f(\bm{x})\}$. 
By passing to a convergent subsequence of $\{\bm{x}^{(t)}\}_t$ if necessary, it is evident to see that, for any $\delta > 0$, there exists $T<+\infty$ such that $\bm{x}^{(T)} \in \mathbb{B}_\delta \left(\bm{x}^{(\infty)}\right)$. Then, by definition, $\bm{x}^{(T)}$ is $(0,\delta)$-GAS, which seems a contradiction to \Cref{thm:findzr} as vanilla subgradient method is clearly deterministic and  general zero-respecting. The subtlety is that \Cref{thm:findzr} rules out any $A \in \mathcal{A}_{\textnormal{det-gzr}}$ with \emph{a priori} finite-time complexity. Such an algorithm needs to promise the same finite $T$ uniformly for all $f \in \mathcal{F}_{C,d}^{\textnormal{Lip}}$. While the result of \citep{davis2020stochastic} implies that for any semi-algebraic $f$ there exists $T<+\infty: \bm{x}^{(T)}$ is $(0,\delta)$-GAS, \Cref{thm:findzr} shows that any \emph{a priori} estimation of  $T$ is impossible.
\end{remark}

\subsection{Proofs}\label{sec:prfs}

For simplicity, we will assume that $L\geq 3, 0\leq \epsilon,\delta < \frac{1}{2}, d \geq 2$, and $\nu>0$ would be a sufficient small quantity whose value may change from line to line.
Besides, we assume that any algorithm starts from $\bm{0}$. Formally, $\forall A \in \mathcal{A}_{\textnormal{det}}, G \in \mathcal{F}_{C,d}^{\textnormal{Lip}}: \bm{x}^{A[G],(1)} = \bm{0}$, which is common in the literature  \cite{carmon2019lower,kornowski2021oracle} and without loss of generality as $\forall G(\bm{x}) \in \mathcal{F}_{C,d}^{\textnormal{Lip}}$, we have  $G(\bm{x} - \bm{x}^{(1)}) \in \mathcal{F}_{C,d}^{\textnormal{Lip}}$.

\subsubsection{The Construction}\label{sec:construction}

\paragraph{Single Coordinate Resisting Function.}

We first adopt a resisting function using the classic resisting oracle of \cite[Theorem 1.1.2]{nesterov2018lectures}. The construction is similar to \citep{zhang2020complexity,vavasis1993black} and we repeat the argument for completeness.
Fix $T < +\infty$ and dimension $d$.
For every query $\bm{x}^{A[F],(t)}\in \mathbb{R}^d, t\in[T]$ from algorithm $A\in\mathcal{A}_{\textnormal{det}}$, the resisting oracle will always return
\[
F\left(\bm{x}^{A[F],(t)}\right) = 0, \qquad \mathcal{O}_F\left(\bm{x}^{A[F],(t)}\right) = x_1 - x_1^{A[F],(t)}.
\]
Then, we show there exists an $F\in \mathcal{F}_{1,d}^{\textnormal{Lip}}$ that is compatible with such a resisting oracle. Collect, reorder, and eliminate duplicate values of $x^{A[F],(t)}_1,\forall t \in [T]$ in increasing order.  Denote the resulting sequence as $\left\{x_1^{(i)}\right\}_{i=1}^{T'}$.
We have $x_1^{(1)}<x_1^{(2)}<\cdots < x_1^{(T')}$.
Let
\[
\sigma \coloneqq \min\left\{ \min_{i\in[T'],j\in[T']} \left|x^{(i)}_1 - x^{(j)}_1\right|, 1\right\}.
\]
For $\bm{x}\in\mathbb{R}^d$, we define $F:\mathbb{R}^d \rightarrow \mathbb{R}$ as
\[\label{eq:F}
F(\bm{x}) \coloneqq     \left\{ \begin{array}{rcl}
         -x_1 + x_1^{(1)} - \frac{\sigma}{2} & \mbox{for} & x_1 \in \left( -\infty, x_1^{(1)} - \frac{\sigma}{4}\right), \\
         x_1 - x^{(t)}_1  & \mbox{for} & x_1 \in \left[ x_1^{(t)} - \frac{\sigma}{4}, \frac{1}{2}\left( x_1^{(t)} + x_1^{(t+1)} \right)  - \frac{\sigma}{4}\right), t \in [T'-1], \\
         -x_1 + x_1^{(t+1)} - \frac{\sigma}{2} & \mbox{for} & x_1 \in \left[ \frac{1}{2}\left( x_1^{(t)} + x_1^{(t+1)} \right)  - \frac{\sigma}{4}, x_1^{(t+1)} - \frac{\sigma}{4}\right), t \in [T'-1], \\
         x_1 - x_1^{(T')} & \mbox{for} &  x_1 \in \left[ x_1^{(T')} - \frac{\sigma}{4}, +\infty\right).
                \end{array}\right. \tag{$\sharp$}
\]
It is easy to see that $F$ is $1$-Lipschitz continuous and $F(\bm{0}) - \inf_{\bm{x}} F(\bm{x}) \leq 1 \leq C$. 
\begin{lemma}\label{lem:Flocal}
For any $t\in[T']$ and $\bm{x} \in \mathbb{B}^1_{\frac{\sigma}{8}}\left(x_1^{(t)}\right) \otimes \mathbb{R}^{d-1}$, it holds that
	$F(\bm{x}) = x_1 - x_1^{(t)}$.
\end{lemma}
\begin{proof} Because for such $\bm{x}$,
	$
	x_1^{(t)} - \frac{\sigma}{4} < x_1^{(t)} - \frac{\sigma}{8} \leq x_1 \leq x_1^{(t)} + \frac{\sigma}{8} < x_1^{(t)} + \frac{\sigma}{4} \leq \frac{1}{2}\left( x_1^{(t)} + x_1^{(t+1)} \right)  - \frac{\sigma}{4}.
	$
\end{proof}
Thus, $F$ is compatible with the resisting oracle by definition of $\left\{x_1^{(i)}\right\}_{i=1}^{T'}$ and \Cref{lem:Flocal}.

\begin{remark}
 For any $\delta > 0$, a deterministic procedure will compute $(0,\delta)$-GAS of $F$ as follows: Query $\bm{x}^{(0)} = \bm{0}$, then query $\bm{x}^{(1)} = \delta\cdot \bm{e}_1$. Then, both $\bm{x}^{(0)}$ and $\bm{x}^{(1)}$ are $(0,\delta)$-GAS of $F$.
 \end{remark}

\paragraph{A ``Wedge'' Replacement.}

\begin{figure}[t]
\centering
  \includegraphics[width=0.45\textwidth]{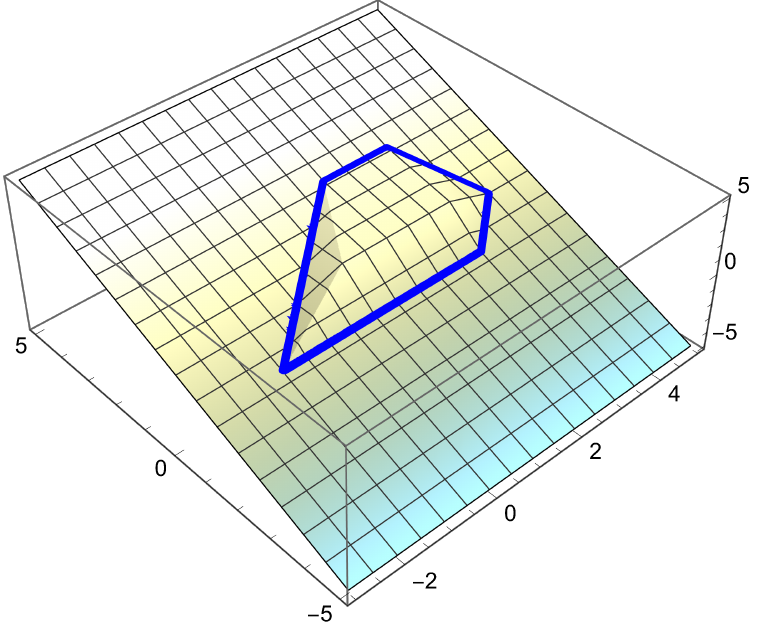}
	\hspace{4em}
  \includegraphics[width=0.35\textwidth]{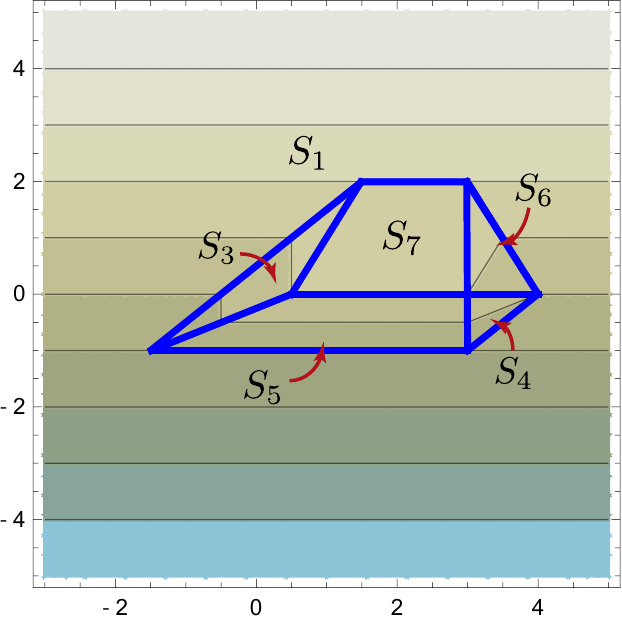}
  \caption{The ``wedge''-shaped resisting function.}\label{fig:wedge}
\end{figure}

In this section, we will build a resisting function with a ``wedge''-like shape.
Our main building block is a $\mathbb{R}^2\rightarrow\mathbb{R}$ ``wedge'' function. For $0<\eta \leq \frac{\sigma}{32}$, we define:
\[
	h(x,y) 
	\coloneqq  \max\left\{ \underbrace{y-\frac{\eta}{2}}_\O1, \textcolor{green!50!black}{ \tilde{h}(x,y) \coloneqq \min\left\{ \underbrace{x + \frac{\eta}{2} }_\O2, \underbrace{2y +\eta\vphantom{\frac{\eta}{2}}}_\O3, \underbrace{\frac{y}{2}+\eta}_\O4 \right\} + \min\left\{\underbrace{-x+\frac{5\eta}{2}}_\O5,\underbrace{- \frac{\eta}{2}
}_\O6\right\} } \right\}.
\]
The following piecewise representation of $h$ is more convenient for analysis.
\[\label{eq:pieceH}
h(x,y) =     \left\{ \begin{array}{rcl}
         y-\frac{\eta}{2} & \mbox{for} & (x,y) \in S_1 \coloneqq \{(x,y): h(x,y)=\O1\}, \\
         3\eta  & \mbox{for} & (x,y) \in S_2\coloneqq \{(x,y): h(x,y)=\O2+\O5\}, \\
         x  & \mbox{for} & (x,y) \in S_3 \coloneqq \{(x,y): h(x,y)=\O2+\O6\}, \\
         -x+2y+\frac{7}{2}\eta  & \mbox{for} & (x,y) \in S_4 \coloneqq \{(x,y): h(x,y)=\O3+\O5\}, \\
         2y+\frac{1}{2}\eta  & \mbox{for} & (x,y) \in S_5 \coloneqq \{(x,y): h(x,y)=\O3+\O6\}, \\
         -x+\frac{y}{2}+\frac{7\eta}{2}  & \mbox{for} & (x,y) \in S_6  \coloneqq \{(x,y): h(x,y)=\O4+\O5\}, \\
         \frac{y}{2}+\frac{\eta}{2}  & \mbox{for} & (x,y) \in S_7 \coloneqq \{(x,y): h(x,y)=\O4+\O6\}. \\
                \end{array}\right. \tag{$\natural$}
\]

The following facts concerning the partitions $\{S_i\}_i$ are useful for further analysis.
\begin{lemma}\label{lem:s2}
	$S_2 = \emptyset$.
\end{lemma}
\begin{proof}
Let us first examine 
\[
S_2=\{(x,y): \O1 \leq (\O2 \wedge \O3 \wedge \O4 ) + (\O5 \wedge \O6), \O2 \leq \O3 \wedge \O4, \O5 \leq \O6 \}.
\]	
Suppose that $(x,y) \in S_2$.
By $\O5 \leq \O6$, we know $x \geq 3 \eta$. Due to $\O2 \leq \O3 \wedge \O4$, we get $x \leq \min \{ \frac{y}{2}, 2y\} + \frac{\eta}{2}$. Thus, it holds that $0<3\eta \leq x \leq \frac{y}{2} + \frac{\eta}{2}$ and $\frac{y}{2} \geq \frac{5\eta}{2}$.
We compute
\[
\O1 - (\O2 \wedge \O3 \wedge \O4 ) - (\O5 \wedge \O6) \geq \O1 - \O4 - \O6 = \frac{y}{2} - \eta \geq \frac{3\eta}{2} > 0,
\]
which gives the contradiction.
\end{proof}

\begin{lemma}\label{lem:s3}
	$S_3 = \left\{(x,y): -\frac{\eta}{2}+y \leq x \leq \frac{\eta}{2} +  \min\left\{2y, \frac{y}{2} \right\} \right\} \subseteq [-\frac{3}{2}\eta, \frac{3}{2}\eta ] \times [-\eta,2\eta]$.
\end{lemma}
\begin{proof}
Note that
\[
S_3=\left\{(x,y): \O1 \leq \O2+ \O6, \O2 \leq \O3 \wedge \O4, \O5 \geq \O6 \right\}.
\]
By $\O2 \leq \O3 \wedge \O4$, we have $x \leq \min \{ \frac{y}{2}, 2y\} + \frac{\eta}{2}$. From $\O1 \leq \O2+ \O6$, it holds $y-\frac{\eta}{2}\leq x$. Then, we have $-\eta \leq y \leq 2\eta$ and $-\frac{3}{2}\eta \leq x \leq \frac{3}{2}\eta$. Thus, the constraint $x \leq 3\eta$ in $\O5\geq \O6$ is always satisfied.
\end{proof}

\begin{lemma}\label{lem:wedge1dprop} It holds that
	 $\partial h(x,y) \subseteq \left[ \begin{array}{c}
    {[-1, 0]} \\ {[\frac{1}{2}, 2]}  \end{array} \right]$ if $(x,y) \in S_3^c$.	
\end{lemma}
\begin{proof}
	By \Cref{lem:s2} and the piecewise characterization of $h$ in \eqref{eq:pieceH}, we know that
	\[
	\nabla h(x,y) \subseteq \left[ \begin{array}{c}
    {[-1, 0]} \\ {[\frac{1}{2}, 2]}  \end{array} \right] \qquad \textnormal{for any} \qquad (x,y) \in \bigcup_{i\neq 3} \intt(S_i).
	\]
	It is easy but tedious to verify that $S_i \cap S_j$ has zero Lebesgue measure for any $i\neq j$ (see \Cref{fig:wedge}), as they are solutions of nondegenerate linear equations.
	Taking a convex hull with \citep[Theorem 9.61]{rockafellar2009variational} and using $\conv(A\times B) = \conv(A) \times \conv( B)$ complete the proof.
\end{proof}

Now, we will proceed to the final construction.
Let $\widetilde{H}:\mathbb{R}^2 \rightarrow \mathbb{R}$ be defined as
\[
\widetilde{H}(x_1,x_2) \coloneqq \max\left\{x_2- \frac{\eta}{2}, \max_{t\in [T']} \tilde{h}\left(x_1 - x^{(t)}_1, x_2\right) \right\},
\]
where $\left\{x_1^{(i)}\right\}_{i=1}^{T'}$ are used to define $F$ in \eqref{eq:F}.
Then, the final hard  ``wedge'' construction $H:\mathbb{R}^d \rightarrow \mathbb{R}$ for $d \geq 2$ is defined as
\[
H(\bm{x}) \coloneqq \max\left\{-5,\widetilde{H}(x_1,x_2) \right\}.
\]

\begin{lemma}\label{lem:strictIneq}
	For any $t\in[T']$, it holds that
	\[
	x_1 - x_1^{(t)} = \tilde{h}\left(x_1 - x_1^{(t)}, x_2\right)
	> x_2 - \frac{\eta}{2} > \max_{t'\in[T']\backslash\{t\}} \tilde{h}\left(x_1 - x_1^{(t')}, x_2\right),\quad \forall \bm{x} \in \mathbb{B}_{\frac{\eta}{8}}^2 \left(\left[ \begin{array}{c}
    x_1^{(t)} \\ 0  \end{array} \right] \right).	
    \]
\end{lemma}
\begin{proof}
	For the first two relations, we note that
	\[
	x_2 - \frac{\eta}{2} \leq - \frac{3\eta}{8} < -\frac{\eta}{8} \leq x_1 - x_1^{(t)} \leq  \frac{\eta}{8} < \frac{\eta}{2} - \frac{\eta}{4} \leq \frac{\eta}{2} + \min\left\{2 x_2, \frac{x_2}{2}\right\},
	\]
	which implies that $\left(x_1 - x_1^{(t)},x_2\right)\in\mathbb{B}_{\frac{\eta}{8}}^2 (\bm{0}) \subset S_3$ by \Cref{lem:s3}. For the last strict inequality, suppose there exists $t' \in [T'], x_1^{(t')}\neq x_1^{(t)}$ such that the opposite holds.
	Then, we have $\O1 \leq(\O2 \wedge \O3 \wedge \O4 ) + (\O5 \wedge \O6)$. By $\O1 \leq(\O3 \wedge \O4 ) +  \O6$, we have $-\eta \leq x_2 \leq 2\eta$. With $\O1 \leq(\O3 \wedge \O4 ) +  \O5$, it holds that $x_1 - x_1^{(t')} \leq \min \{x_2, -\frac{x_2}{2}\} + 2\eta \leq 4\eta$. Due to $\O1 \leq \O2 + \O6$, we know $x_1 - x_1^{(t')} \geq x_2 - \frac{\eta}{2} \geq -\frac{3}{2}\eta$, which implies that $\left| x_1 - x_1^{(t')} \right|\leq 4\eta$. However,
	\[
	\left| x_1 - x_1^{(t')} \right| \geq \left| x_1^{(t)} - x_1^{(t')} \right| - |x_1| \geq \sigma - \frac{\eta}{8} > 30\eta,
	\]
	which gives the contradiction.

\end{proof}

The main lemma in this part is as follows:
\begin{lemma}\label{lem:HisF}
	The following hold.
	\begin{itemize}
		\item $H$ is $3$-Lipschitz continuous and $H(\bm{0}) - \inf_{\bm{x}} H(\bm{x}) \leq 6$.
		\item There exists a $\nu>0$ such that
		\[
		F(\bm{y}) = H(\bm{y}),\qquad \forall \bm{y} \in \bigcup_{t \in [T']} \mathbb{B}_\nu^2 \left(\left[ \begin{array}{c}
    x_1^{(t)} \\ 0  \end{array} \right] \right) \otimes \mathbb{R}^{d-2}.
		\]
	\end{itemize}
\end{lemma}
\begin{proof} From the piecewise linear expression in \eqref{eq:pieceH}, it is easy to see that $H$ is $3$-Lipschitz.
Note that $h(x,0) \leq \max\{-\frac{\eta}{2},\eta - \frac{\eta}{2}\} = \frac{\eta}{2} \leq 1, \forall x \in \mathbb{R}$. Then, it follows that $H(\bm{0}) - \inf_{\bm{x}} H(\bm{x})  = \max\{-5, \widetilde{H}(0,0)\} + 5 \leq 1 + 5 = 6$.
Let $\nu = \frac{\eta}{8}$.
	\Cref{lem:strictIneq} new yields that $\widetilde{H}(\bm{y}) > y_2 - \frac{\eta}{2} \geq -\frac{5\eta}{8} > -5$. Thus, with \Cref{lem:Flocal} and $\nu < \frac{\sigma}{8}$, for all such $\bm{y}$ we have $F(\bm{y}) = \widetilde{H}(\bm{y}) = H(\bm{y})$.
\end{proof}

\paragraph{Resolution of Approximate Stationarity.}

In this part, we will prove that there is no GAS point below certain precision near the  ``ambiguity region''.
\begin{lemma}\label{lem:HisHtilt}
	If $H(\bm{x}) \geq -1$, then for any $\bm{y}\in\mathbb{B}^d_{2\delta}(\bm{x})$, we have $H(\bm{y}) = \widetilde{H}(y_1, y_2)$.
\end{lemma}
\begin{proof} By \Cref{lem:HisF} and $0 \leq \delta < \frac{1}{2}$, we know that for any $\bm{y}\in\mathbb{B}^d_{2\delta}(\bm{x})$,
	\[
	H(\bm{y}) \geq H(\bm{x}) - \left| H(\bm{y}) - H(\bm{x}) \right|
	\geq -1 - 6\delta \geq -4 > -5,
	\]
	as required.
\end{proof}

The main lemma in this part is as follows:
\begin{lemma}\label{lem:noGAS-4}
	If $H(\bm{x}) \geq -1$, then $\dist\Big( 0, \partial_\delta H(\bm{x}) \Big)\geq \frac{1}{\sqrt{17}}$.
\end{lemma}
\begin{proof}
We begin by computing
\begin{align*}	
\partial_\delta H(\bm{x}) &= \conv\left({\textstyle \bigcup_{\bm{y} \in \mathbb{B}^d_\delta (\bm{x})} \partial H (\bm{y})} \right) \tag{\Cref{def:subg}} \\
&= \conv\left({\textstyle \bigcup_{(y_1,y_2) \in \mathbb{B}^2_\delta \big((x_1,x_2)\big)} \partial \widetilde{H} (y_1,y_2)\otimes \{0\}^{\otimes d-2}} \right) \tag{\Cref{lem:HisHtilt}} \\
&= \conv\left(\left({\textstyle \bigcup_{(y_1,y_2) \in \mathbb{B}^2_\delta \big((x_1,x_2)\big)}} \partial \widetilde{H}(y_1,y_2)\right)\otimes \{0\}^{\otimes d-2} \right) \tag{\cite[\S 3, Exercise 3(4)]{blyth1975set}} \\
&= 
\conv\left({\textstyle \bigcup_{(y_1,y_2) \in \mathbb{B}^2_\delta \big((x_1,x_2)\big)}} \partial \widetilde{H} (y_1,y_2)\right)\otimes \{0\}^{\otimes d-2} \\
& = \partial_\delta \widetilde{H}(x_1,x_2) \otimes \{0\}^{\otimes d-2}.
\end{align*}
Therefore, it suffices to show that $\dist\Big(0, \partial_\delta \widetilde{H}(x_1,x_2)\Big) \geq \frac{1}{\sqrt{17}}$. Let $\bm{g} \coloneqq \arg\min_{\bm{z}\in \partial_\delta \widetilde{H}(x_1,x_2)} \|\bm{z}\|$. By Carathéodory’s theorem \cite[Theorem 2.29]{rockafellar2009variational}, we can write $\bm{g}$ with a finite convex combination 
\[
\bm{g} = \sum_{i=1}^{3} \lambda_i \bm{g}^i, \qquad \textnormal{where } \  \bm{g}^i \in \partial \widetilde{H}(\bm{y}^i), \  \bm{y}^i \in \mathbb{B}^2_\delta (x_1,x_2),\  \sum_{i=1}^{3}\lambda_i = 1,\  \forall i \in [3]: \lambda_i \geq 0.
\]
Consider partition of $[3]=P_1 \sqcup P_2 \sqcup P_3$ with
\[
\begin{aligned}
P_1&\coloneqq\left\{i\in[3]: \exists t \in [T'], \bm{y}^i \in \intt\big(S_3-(x_1^{(t)},0)\big)\right\}, \\
P_2&\coloneqq\left\{i\in[3]: \forall t \in [T'], \bm{y}^i \in \big(S_3-(x_1^{(t)},0)\big)^c\right\}, \\
P_3&\coloneqq\left\{i\in[3]: \exists t \in [T'], \bm{y}^i \in \bd\big(S_3-(x_1^{(t)},0)\big)\right\}.
\end{aligned}
\]
Then, we can rewrite $\bm{g}$ by averaging within $P_i, \forall i \in [3]$:
\[
\bm{g}  =  \sum_{i \in P_1} \lambda_i \bm{g}^i + \sum_{j \in P_2} \lambda_j \bm{g}^j + \sum_{k \in P_3} \lambda_k \bm{g}^k = \sum_{i=1}^{3} \theta_i \bm{g}^{P_i},
\]
with $\theta_i \coloneqq \sum_{j \in P_i} \lambda_j$ and $\bm{g}^{P_i} \coloneqq \sum_{j \in P_i} \frac{ \lambda_j }{\theta_i} \bm{g}^j$ for any $i \in [3]$.
Thus, it suffices to consider taking the convex hull within every $P_i, \forall i \in [3]$ according the following:

\begin{itemize}
	\item Averaging within $P_1$: $\bm{g}^{P_1}=\bm{e}_1$ by the piecewise characterization in \eqref{eq:pieceH}.
	\item Averaging within $P_2$: $\bm{g}^{P_2} \subseteq \left[ \begin{array}{c}
    {[-1, 0]} \\ {[\frac{1}{2}, 2]}  \end{array} \right]$ by \Cref{lem:wedge1dprop}.
    \item Averaging within $P_3$: $\bm{g}^{P_3}
\subseteq 
\conv \left( \bm{e}_1, \left[ \begin{array}{c}
    {[-1, 0]} \\ {[\frac{1}{2}, 2]}  \end{array} \right] \right)$ by \citep[Theorem 9.61]{rockafellar2009variational}.
\end{itemize}
Taking a convex combination, we can assert that
\[
\bm{g}=\sum_{i=1}^3 \theta_i \bm{g}^{P_i} \subseteq \conv \left( \bm{e}_1, \left[ \begin{array}{c}
    {[-1, 0]} \\ {[\frac{1}{2}, 2]}  \end{array} \right] \right).
\]
What is left is to show the following numerical estimation:
\begin{lemma}\label{lem:numineq1} It holds that
\[
	\frac{1}{17}= \min_{t,v_1,v_2}   (t+(1-t)v_1)^2 + (1-t)^2v_2^2\quad \textnormal{ s.t.} \quad t\in[0,1],v_1\in[-1,0],v_2 \in [1/2,2].
\]
\end{lemma}
\begin{proof}
	It is easy to see that $v_2^* = \frac{1}{2}$. Let the objective function be $Q$.	
	Note that
	\[
	Q(t,v_1,v_2^*) = 
	\left(\frac{1}{4} + (v_1-1)^2\right)\cdot t^2 - \frac{1}{4}\left(  v_1 - \frac{1}{2}\right)^2\cdot t + v_1^2 + \frac{1}{4}.
	\]
	By first-order optimality condition and $v_1 \in [-1,0]$, we have
	\[
	0<t^* = \frac{(2v_1-1)^2}{1+4(v_1 -1)^2} = 1-\frac{4(1-v_1)}{1+4(1-v_1)^2}<1.
	\]
	This implies that
	\[
	Q(t^*,v_1,v_2^*)=\frac{1}{1+4(v_1 -1)^2} \geq \frac{1}{17},
	\]
	as required.
\end{proof}

We continue with
\begin{align*}
	\|\bm{g}\| &\geq \dist\left(0, \conv \left( \bm{e}_1, \left[ \begin{array}{c}
    {[-1, 0]} \\ {[\frac{1}{2}, 2]}  \end{array} \right] \right) \right) \\
	&= \min_{t \in [0,1]} \left\| \left[ \begin{array}{c}
    {t + (1-t)\cdot [-1, 0]} \\ {(1-t)\cdot [\frac{1}{2}, 2]}  \end{array} \right] \right\|   \\
	&=\min_{\substack{t:0 \leq t \leq 1,\\v_1:-1 \leq v_1\leq 0,\\v_2:\frac{1}{2} \leq v_2 \leq 2}} \sqrt{ (t+(1-t)v_1)^2 + (1-t)^2v_2^2} \\
	&= \frac{1}{\sqrt{17}}. \tag{by \Cref{lem:numineq1}}
\end{align*}
Thus, we have proved that $\dist\big(0, \partial_\delta H(\bm{x})\big) = \dist\left(0, \partial_\delta \widetilde H(x_1,x_2)\right) = \|\bm{g}\| \geq \frac{1}{\sqrt{17}}$.
\end{proof}

\subsubsection{Hardness Results}\label{sec:zrDetReduction}

In this subsection, we put everything together. We will first prove \Cref{thm:findzr} as its proof is conceptually easier and can be reused in that of \Cref{thm:find}.
\paragraph{Deterministic General Zero-Respecting Algorithms.}
\begin{proof}[Proof of \Cref{thm:findzr}]
	Fix any $A\in\mathcal{A}_{\textnormal{det-gzr}}, d \geq 2$, and a finite iteration number $T < +\infty$. 
	Apply $A$ to the single coordinate resisting construction in \Cref{sec:construction}. We get a resisting $F:\mathbb{R}^d\rightarrow \mathbb{R}$ and $\left\{\bm{x}^{A[F],(t)}\right\}_{t=1}^T$. Recall that $\bm{x}^{A[F],(1)} = \bm{0}$. By \Cref{lem:Flocal} and the definition of $\mathcal{A}_{\textnormal{det-gzr}}$, with a simple induction on $t$, we have $2 \notin \supp\left(\bm{x}^{A[F],(t)} \right), \forall t \in [T]$. With the ``wedge'' construction $H:\mathbb{R}^d \rightarrow \mathbb{R}$ according to $F$
 and \Cref{lem:HisF}, we know that $H$ and $F$ are indistinguishable to $A$ by querying the local oracle on $\left\{\bm{x}^{A[F],(t)}\right\}_{t=1}^T$. Formally, there exists a $\nu > 0$ such that $\mathcal{O}_F\left( \bm{x}^{A[F],(t)} \right)(\bm{y}) = \mathcal{O}_H\left( \bm{x}^{A[F],(t)} \right)(\bm{y})$ for all $\bm{y} \in \mathbb{B}^d_\nu\left( \bm{x}^{A[F],(t)} \right),t\in[T]$.  That is to say,  $\bm{x}^{A[F],(t)} = \bm{x}^{A[H],(t)}, \forall t \in [T]$ as $A$ is deterministic.
	Thus, for any $t\in[T]$, it holds that $H\left(\bm{x}^{A[H],(t)}\right) = F\left(\bm{x}^{A[F],(t)}\right) = 0 > -1$. However, by \Cref{lem:noGAS-4}, 
	\[
	\min_{t \in [T]} \ \dist\Big(0, \partial_\delta H\left(\bm{x}^{{A[H]},(t)} \right)\Big) \geq  \frac{1}{\sqrt{17}} > \epsilon,
	\]
	which completes the proof by noting that $H \in \mathcal{F}_{C,d}^{\textnormal{Lip}}$ from \Cref{lem:HisF}.
\end{proof}

\paragraph{Deterministic Algorithms.}

For the general deterministic case, we use the classic adversarial rotation argument \cite{nemirovskij1983problem,carmon2019lower,woodworth2016tight} to reduce it to the $\mathcal{A}_{\textnormal{det-gzr}}$ case. 
\begin{proof}[Proof of \Cref{thm:find}]

Fix any $A\in\mathcal{A}_{\textnormal{det}}$, a finite $T < +\infty$, and $d \geq T+1$. Apply $A$ to the single coordinate resisting construction in \Cref{sec:construction}. We get a resisting $F:\mathbb{R}^d\rightarrow \mathbb{R}$ and $\left\{\bm{x}^{A[F],(t)}\right\}_{t=1}^T$. Recall that $\bm{x}^{A[F],(1)} = \bm{0}$. Let 
$
\bm{V}\coloneqq  \left[ \begin{array}{cccc}
     \bm{e}_1 & \bm{x}^{A[F],(2)} &  \cdots &	 \bm{x}^{A[F],(T)} \end{array} \right]
      \in \mathbb{R}^{d\times T}.
$
Let $\bm{u}_2 \in \k(\bm{V}^\top)$ and $\|\bm{u}_2\| = 1$, which is possible due to $d > T$. 
By choosing $\widetilde{\bm{U}}\in\mathbb{R}^{d\times (d-2)}$ as an orthonormal basis for the orthogonal complement of $\spn\{ \bm{e}_1, \bm{u}_2 \}$, we define an orthonormal
$
\bm{U}\coloneqq  \left[ \begin{array}{ccc}
     \bm{e}_1 & \bm{u}_2 & \widetilde{\bm{U}}  \end{array} \right]
      \in \mathbb{R}^{d\times d}.
$
Now, let $H$ be the ``wedge'' construction according to $F$ and $G(\bm{x})\coloneqq H(\bm{U}^\top \bm{x})$. We aim to show that there exists a $\nu > 0$ such that $\mathcal{O}_F\left( \bm{x}^{A[F],(t)} \right)(\bm{y}) = \mathcal{O}_G\left( \bm{x}^{A[F],(t)} \right)(\bm{y})$ for all $\bm{y} \in \mathbb{B}^d_\nu\left( \bm{x}^{A[F],(t)} \right),t\in[T]$.
To this end, fix $t \in [T]$ and $\bm{y} \in \mathbb{B}^d_\nu\left( \bm{x}^{A[F],(t)} \right)$. By \Cref{lem:Flocal} with $0<\nu < \frac{\sigma}{8}$, we know that $F(\bm{y}) = y_1 - x^{A[F],(t)}_1$. Observe that $\bm{U}^\top\bm{y} \in \mathbb{B}_\nu^d\left(\bm{U}^\top\bm{x}^{A[F],(t)} \right)$ as $\left\| \bm{U}^\top\left(\bm{y} -\bm{x}^{A[F],(t)}\right) \right\| \leq \nu$. Recall that $2 \notin \supp\left(\bm{U}^\top\bm{x}^{A[F],(t)} \right),\forall t \in [T]$ by construction of $\bm{U}$. It follows from \Cref{lem:HisF} that $H\left(\bm{U}^\top\bm{y}\right) = F\left(\bm{U}^\top\bm{y}\right)$. Note that $\bm{U}^\top\bm{y} \in \mathbb{B}_\nu^1\left(x^{A[F],(t)}_1\right)\otimes \mathbb{R}^{d-1}$. Using \Cref{lem:Flocal} again yields
\[
G(\bm{y}) = H(\bm{U}^\top \bm{y}) = F(\bm{U}^\top \bm{y}) = \bm{e}_1^\top \bm{y} - \bm{e}_1^\top \bm{x}^{A[F],(t)} =  y_1 - x^{A[F],(t)}_1 = F(\bm{y}).
\]
Thus, by shrinking $\nu$ if necessary, we can see that $G$ and $F$ are indistinguishable to $A$ by querying the local oracle at $\left\{\bm{x}^{A[F],(t)}\right\}_{t=1}^T$. It holds that $\bm{x}^{A[F],(t)} = \bm{x}^{A[G],(t)}, \forall t \in [T]$ as $A$ is deterministic. 
 Besides, we observe that
\begin{align*}
\partial_\delta G\left(\bm{x} \right) 
&=\conv\Big({\textstyle \bigcup_{\bm{y} \in \mathbb{B}_\delta^d(\bm{x})}} \bm{U}^\top\partial H\left(\bm{U}^\top\bm{y}\right)\Big)  \\
&=\bm{U}^\top\conv\Big({\textstyle \bigcup_{\bm{y} \in \mathbb{B}_\delta^d(\bm{x})}} \partial H\left(\bm{U}^\top\bm{y}\right)\Big)  \\
&=\bm{U}^\top\conv\Big({\textstyle \bigcup_{\bm{z} \in \mathbb{B}_\delta^d(\bm{U}^\top\bm{x})} }\partial H\left(\bm{z}\right)\Big) 
=\bm{U}^\top \partial_\delta H\left(\bm{U}^\top \bm{x} \right),
\end{align*}
where the first equality is by \citep[Theorem 8.49, Exercise 10.7]{rockafellar2009variational} (see also \citep[Theorem 2.3.10]{clarke1990optimization}), the second one can be deduced from the bijectivity of $\bm{U}$, \citep[Ch.\ 1, \S2, Exercise 2(b)]{munkres2000topology} and \citep[\S A, Proposition 1.3.4]{hiriart2004fundamentals}, and the third equality is due to $\left\{\bm{U}^\top\bm{y}:\bm{y}\in\mathbb{B}_\nu^d(\bm{x})\right\} = \mathbb{B}_\nu^d\left(\bm{U}^\top\bm{x}\right)$.
Therefore, for any $t\in [T]$, $H\left(\bm{U}^\top\bm{x}^{A[G],(t)}\right)= F\left(\bm{x}^{A[F],(t)}\right) = 0 > -4$, we conclude by \Cref{lem:noGAS-4} that,
\[
\min_{t \in [T]} \ \dist\Big(0, \partial_\delta G\left(\bm{x}^{{A[G]},(t)} \right)\Big) = \min_{t \in [T]} \ \dist\Big(0, \bm{U}^\top\partial_\delta H\left(\bm{U}^\top\bm{x}^{{A[G]},(t)} \right)\Big) \geq   \frac{1}{\sqrt{17}} > \epsilon,
\]
which completes the proof by noting that $G \in \mathcal{F}_{C,d}^{\textnormal{Lip}}$ from a simple corollary of \Cref{lem:HisF}.
\end{proof}

\section{Concluding Remarks}\label{sec:concl}
Recently, \citet{zhang2020complexity} introduced a randomized algorithm that computes Goldstein's approximate stationarity \citep{goldstein1977optimization} to arbitrary precision with a dimension-free polynomial oracle complexity. 
	In this paper, we show that no deterministic algorithm can do the same.
	Even without the dimension-free requirement, we show that any finite time guaranteed deterministic method cannot be general zero-respecting, which rules out most of the oracle-based methods in smooth optimization and any trivial derandomization of \citet{zhang2020complexity}. 
	It also implies that any finite-time deterministic method for GAS must be significantly different from most of the commonly used algorithmic scheme in smooth optimization. Thus, new algorithmic ideas are necessary for computing GAS in finite time.
	Our results shed light on a fundamental hurdle of nonconvex nonsmooth problems in the modern large-scale setting and their infinite-dimensional extension.

\bibliography{gas-hard.bib}
\bibliographystyle{plainnat}

\end{document}